\documentclass[a4paper,twoside,10pt]{article}
\usepackage{lmodern}
\usepackage[T1]{fontenc}
\usepackage{mathtools}
\usepackage{inputenc}
\usepackage{amsmath}
\usepackage{amsthm}
\usepackage{amsfonts}
\usepackage{graphicx}
\usepackage{color} 
\usepackage[all]{xy}
\usepackage{multicol}
\usepackage{amssymb}
\usepackage{setspace}
\usepackage{makeidx}
\usepackage{fancybox, calc}
\usepackage{tikz}  
\usepackage{eso-pic}
\usepackage{pdflscape}
\usepackage[some]{background}
\usepackage{enumitem}
\usepackage{tikz}  

\theoremstyle{definition}
\newtheorem{defin}{Definition}
\newtheorem{ejem}{Example}

\theoremstyle{definition}
\newtheorem{teo}{Theorem}
\newtheorem{cor}{Corollary}
\newtheorem{prop}{Proposition}
\newtheorem{lema}{Lemma}

\theoremstyle{remark}
\newtheorem{rem}{Remark}

{%\setlength{\parindent}{0pt}
	
\def\ZZ{\mathbb{Z}}

\def\RR{\mathbb{R}}
\def\PP{\mathbb{P}}

\def\FF{\mathcal{F}}
\def\NN{\mathcal{N}}
\def\UU{\mathcal{U}}

\def\PP{\mathcal{P}}
\def\HH{\mathcal{H}}
\def\CC{\mathcal{C}}
\def\LL{\mathcal{L}}
\def\KK{\mathcal{K}}

\def\DD{\mathcal{D}}

\def\dx1{\dfrac{dx_1}{x_1}}
\def\dy1{\dfrac{dy_1}{y_1}}

\makeindex         
	
\title{Combinatorial Aspects of Classical Resolution of Singularities}
\author{Beatriz Molina-Samper}
\date{}            
	                    
\begin{document}                     
\maketitle  
\begin{center}
	\emph{To Professor Felipe Cano on the occasion of his 60th birthday}
\end{center}
\begin{abstract}
	We describe combinatorial aspects of classical resolution of singularities that are free of characteristic and can be applied to singular foliations and vector fields as well as to functions and varieties. In particular, we give a combinatorial version of Hironaka's maximal contact theory in terms of characteristic polyhedra systems and we show the global existence of maximal contact in this context.
		
%\keywords{Polyhedra systems \and Maximal contact \and Reduction of singularities}
	\end{abstract}
\section{Introduction}       
	
	We present here a combinatorial formulation for the procedure of reduction of singularities in terms of polyhedra systems. This combinatorial structure is free of restrictions on the characteristic and provides a combinatorial support for the reduction of singularities of varieties, foliations, vector fields and differential forms, among other possible objects.
	
	Hironaka's characteristic polyhedra represent the combinatorial steps in almost any procedure of reduction of singularities. This is implicit in the formulation of the polyhedra game \cite{Hir}, solved by Spivakovsky \cite{Spi}, and in many other papers about characteristic polyhedra \cite{Cos2,Cos3,Sch,Spi,You}. 
	
	The combinatorial features concerning the problems of reduction of singularities are reflected in polyhedra systems without loosing the global aspects. In particular, we need to project the problem over a "Maximal Contact Support Fabric", that plays the role of a maximal contact variety \cite{Aro-H-V}. The proof of global existence of combinatorial maximal contact is also a problem of the same nature as the reduction of singularities. We present here a proof and it is solved thanks to the induction hypothesis on the dimension. 
	
	Let us remark that in positive characteristic maximal contact does not necessarily exist \cite{Cos,Cos4,Cos-J-S,Nar}. Nevertheless we do find it in the combinatorial situations.
	
	The polyhedra systems have evident links with toric varieties and they also provide a combinatorial global support for the so-called Newton non-degenerated varieties \cite{Oka}. In a forthcoming paper we plan to develop these last two aspects also for codimension one singular foliations.          
          
\section{Support Fabrics for Polyhedra Systems}       
The support fabrics play the role of ambient spaces for supporting polyhedra systems. They express the stratified structure of the space.

Let $I$ be a non-empty finite set and denote by $\PP(I)$ the set of subsets of $I$. We consider the \emph{Zariski topology on $\PP(I)$} whose closed sets are the sets $\KK \subset \PP(I)$ having the property: 
$$\text{ If } J_1 \in \KK \text{ and } J_2 \supset J_1 \Rightarrow J_2 \in \KK.$$     
A subset $\HH \subset \PP(I)$ is open if and only if for every $J \in \HH$ we have  $\PP(J) \subset \HH$. The closure of  $\{J\} \subset \PP(I)$ is the set $\overline{\{J\}}=\{J' \subset I ; \; J' \supset J\}$. 
      
A \emph{support fabric} is a pair  $\FF=(I,\HH)$ where $I$ is a finite set and $\HH$ is an open set for the Zariski topology on $\PP(I)$. The elements of $\HH$ are called \emph{strata} and $(I,\leq)$ is called the \emph{index set}. The \emph{dimension} $\text{dim} (\FF)$ is defined by $\text{dim} (\FF)=\max \{\# J; \; J \in \HH\},$ where $\#J$ denotes the number of elements of $J$. Note that the strata $J \in \HH$ where the dimension is reached are closed points in $\HH$.

The first example is the \emph{local support fabric} $\LL_I$ associated to the index set $I$, defined by $\LL_I=(I,\PP(I))$. We can obtain other ones as follows: 
\begin{itemize}
	\item The \emph{restriction $\FF|_{\UU}$} of the support fabric $\FF$ to an open set $\UU \subset \HH$. It is given by $\FF|_{\UU}=(I,\UU)$. Note that we always have $\FF=\LL_I|_{\HH}$.
	\item The \emph{reduction} $\text{Red}_{\KK}(\FF)$ of the support fabric $\FF$ to a closed subset $\KK$ of $\HH$, that is $\KK$ is the intersection with $\HH$ of a closed set of $\PP(I)$. We consider the smallest open set $\HH_{\KK}$ that contains $\KK$. That is $\HH_{\KK}=\bigcup_{J \in \KK}\PP(J)$.
	Then, we define $\text{Red}_{\KK}(\FF)=\FF|_{\HH_{\KK}}$.
	\item \label{ex:proyeccion} The \emph{projection} $\FF^T$ from a non-empty stratum $T \in \HH$. The set $$\HH^T=\{J \setminus T; \; T \subset J,\, J \in \HH\}$$ 
	is open in $\PP(I^T)$, where $I^T=I \setminus T$. We define $\FF^T=(I^T,\HH^T)$. Note that $\text{dim}(\FF^T) < \text{dim}(\FF)$.
\end{itemize}

Given a support fabric $\FF=(I, \HH)$, the \emph{relevant index set $\mathcal{I}_{\FF}$ of $\FF$} is defined by $\mathcal{I}_{\FF}= \cup_{J \in \HH}J$. We say that two support fabrics $\FF_1=(I_1,\HH_1)$ and $\FF_2=(I_2,\HH_2)$ are \emph{equivalent} if there is a bijection $\phi:\mathcal{I}_{\FF_1} \rightarrow \mathcal{I}_{\FF_2}$ such that $\HH_2= \{\phi(J); \, J \in \HH_1\}$. In this case we say that $\phi$ is an \emph{equivalence between $\FF_1$ and $\FF_2$}.

Let us define the \emph{blow-up $\pi_J(\FF)$ of $\FF$ centered in a non-empty stratum $J \in \HH$}. Take $I'=I \cup \{\infty\}$ where $\infty \notin I$. Define $\HH'=\HH_s' \cup \HH_{\infty}'$ where $\HH_s'=\HH\setminus \overline{\{J\}}$ and                  
$$\HH'_{\infty}= \bigcup\limits_{K\in \overline{\{J\}}\cap \HH} \HH'^{K}_{\infty}; \text{ where } \HH^{'K}_{\infty}=\{(K\setminus J) \cup A \cup \{\infty\}; \; A \subsetneq J\}.$$
We have that $\HH'$ is open in $\PP(I')$ and we define $\pi_J(\FF)=(I',\HH')$.
Note that for every $K\in \overline{\{J\}}\cap \HH$, there is a bijection $ \HH^{'K}_{\infty}\rightarrow \PP(J) \setminus \{J\}$ given by $J' \mapsto A_{J'}$, where $A_{J'}$ is the unique subset of $J$ such that $$J'=(K\setminus J) \cup A_{J'} \cup \{\infty\}.$$                

\begin{rem}        
	The dimension of a support fabric $\FF$ is invariant by blow-ups.
\end{rem}
\begin{rem}
	The relevant index sets after the blow-up $\pi_J(\FF)$ are given by
	$$ \mathcal{I}_{\pi_J(\FF)}=\left\{
	\begin{array}{ccc} 
	\mathcal{I}_{\FF} \cup \{\infty\}, & \text{if} & \#J \geq 2; \\
	\mathcal{I}_{\FF} \setminus\{j\} \cup \{\infty\}, & \text{if} & J=\{j\}. 
	\end{array}
	\right.$$
Moreover, if there is an equivalence $\phi$ between two support fabrics $\FF_1$ and $\FF_2$, then for every $J \in \HH_1$, the blow-ups $\pi_J(\FF_1)$ and $\pi_{\phi(J)}(\FF_2)$ are also equivalent.
\end{rem}
            
There is a surjective map $\pi_J^\#:\HH' \rightarrow \HH$ between the strata sets of $\pi_J(\FF)$ and $\FF$ respectively, given by  $\pi_J^{\#}(\HH'^{K}_{\infty})=\{K\}$, for each $K \in  \overline{\{J\}}\cap \HH$ and by $\pi_J^{\#}(J')=J'$, if $J' \in \HH'_s=\HH\setminus \overline{\{J\}}$. 
\begin{prop}       
	The map $\pi_J^{\#}:\HH' \rightarrow \HH$ is continuous.
	\begin{proof}     
		Remark that for every $J' \in \HH'$, we have $\pi_J^{\#}\Big(\PP(J')\Big) \subset \PP\left(\pi_J^{\#}(J')\right)$.   
	\end{proof}       
\end{prop}        
\begin{cor} \label{cor:continuidad}
	Given an open set $\UU \subset \HH$, we have $\pi_J(\FF|_{\UU})=\pi_J(\FF)|_{\UU'}$, where $\UU'=(\pi_J^{\#})^{-1}(\UU)$.
 
\end{cor}         
 
\begin{ejem} \label{ex:fabric}
Consider a pair $(M,E)$, where $M$ is a complex analytic variety and $E$ is a \emph{strong normal crossings divisor} of $M$. That is, $E$ is the union of a finite family $\{E_i\}_{i \in I}$ of irreducible smooth hypersurfaces $E_i$, where we fix an order in the index set $I$ and the next properties hold: 
\begin{enumerate}
\item Given a point $p \in M$ and the subset $I_p=\{i \in I; \; p \in E_i\}$, there is part of a local coordinate system in $p$ of the form $\{x_i\}_{i \in I_p}$ in such a way that $E_i=(x_i=0)$ locally in $p$ for each $i \in I_p$ (such coordinate systems are called \emph{adapted to $E$}). 
\item The non-singular closed analytic set $E_J=\bigcap_{j\in J} E_j$ is connected for each $J \subset I$.
\end{enumerate}
The pair $\FF_{M,E}=(I,\HH)$ is a support fabric, where $\HH=\{J \subset I; \; E_{J} \neq \emptyset\}$. Note that the dimension of $\FF_{M,E}$ is not necessarily the dimension of $M$. Given $J \in \HH$, we can perform the usual blow-up 
$$\pi:(M',E') \rightarrow (M,E)$$
centered in $E_J$, where $E'$ is the total transform of $E$, that is, $E'=\pi^{-1}(E)$. In this situation we have $\pi_J(\FF)=\FF_{M',E'}$.
\end{ejem}  
           
\section{Polyhedra Systems}
When we perform a blow-up of the ambient space, it becomes of global nature. Polyhedra systems provide a way of giving Newton or characteristic polyhedra in a coherent way along global ambient spaces, represented in our case by the support fabric.
         
\subsection{Definitions.}          
Given a totally ordered finite set $J$, we recall that $\RR^J$ denotes the set of maps $\sigma:J \rightarrow \RR$. If $J_1 \subset J_2$, there is a canonical projection $\mbox{\rm pr}_{J_2,J_1}:\RR^{J_2} \rightarrow \RR^{J_1}$ given by $\sigma \mapsto \sigma_{|_{J_1}}$.
For a subset $A \subset \RR_{\geq 0}^J$, we define the \emph{positive convex hull} $[[A]]$ by
$$[[A]]=\mbox{Convex hull}\Big(A+\RR_{\geq 0}^J\Big) \subset \RR^J_{\geq 0}.$$
Let $d \in \ZZ_{>0}$ be a positive integer. We say that a subset $\Delta \subset \RR^J_{\geq 0}$ is a \emph{characteristic polyhedron with denominator $d$} if there is $A \subset (1/d)\ZZ_{\geq 0}^J$ such that $\Delta=[[A]]$.

\begin{defin}  
A \emph{polyhedra system $\DD$ over a support fabric $\FF=(I,\HH)$ with denominator $d$}, is a triple $\DD= (\FF ; \{\Delta_J\}_{J\in \HH},d)$, where each $\Delta_J \subset \RR^J_{\geq 0}$ is a characteristic polyhedron with denominator $d$, in such a way that for every $J_1,J_2 \in \HH$ where $J_1 \subset J_2$, we have $\Delta_{J_1}=pr_{J_2,J_1}(\Delta_{J_2})$. When $d=1$, we say that $\DD$ is a \emph{Newton polyhedra system} over $\FF$.
If it is necessary, we denote by $\FF_\DD$ the support fabric of $\DD$. We define the \emph{dimension} of $\DD$ as the dimension of $\FF_{\DD}$.
\end{defin}
We say that two polyhedra systems $\DD_1$ and $\DD_2$ are \emph{equivalent} if there is an equivalence $\phi$ between $\FF_1$ and $\FF_2$ such that $\Delta^2_{\phi(J)}=\Delta^1_J$ for all $J\in \HH_1$, where $\DD_1=(\FF_1, \{\Delta^{1}_J\}_{J \in \HH_1},d)$ and $\DD_2=(\FF_2, \{\Delta^2_J\}_{J \in \HH_2},d)$. We also say that $\phi$ is an \emph{equivalence between $\DD_1$ and $\DD_2$}.
 
Given a polyhedra system $\DD=(\FF, \{\Delta_J\}_{J \in \HH},d)$, there is a \emph{Newton polyhedra system $\NN(\DD)$ associated to $\DD$}, defined by
$$ \NN(\DD)=(\FF; \{d\Delta_J\}_{J\in \HH},1).$$
Conversely, given another positive integer number $d'$, we can obtain a new polyhedra system $\DD /d'$ given by
$ \DD/d' =\left(\FF; \left\{(1/d')\Delta_J\right\}_{J\in \HH},dd'\right).$ In particular, we have $\NN(\DD)/d =\DD$. 

Let us see some examples of polyhedra systems:
\begin{itemize}
	\item Given a characteristic polyhedron $\Delta \subset \RR^I_{\geq 0}$ with denominator $d$, we define the \emph{local polyhedra system $\LL (\Delta,d)$} by $$\LL (\Delta,d)=(\LL_I; \{\Delta_J\}_{J\in \PP(I)},d), \text{ where } \Delta_J=pr_{IJ}(\Delta).$$
	\item The \emph{restriction $\DD|_{\UU}$} of $\DD$ to an open set $\UU \subset \HH$ is defined by 
	$$\DD|_{\UU}=(\FF|_{\UU}; \{\Delta_J\}_{J\in \UU},d).$$
	\item The \emph{reduction $\text{Red}_{\KK}(\DD)$} of $\DD$ to a closed set $\KK$ of $\HH$ is  $\mbox{\rm Red}_{\KK}(\DD)=\DD|_{\HH_{\KK}}$.
	\item The \emph{fitting} $\widetilde{\DD}$ of $\DD$ is given as follows. For a subset $A \subset \RR^I_{ \geq 0}$, we consider the \emph{fitting vector $w_A \in \RR_{\geq 0}^I$} defined by $w_A(j)=\min\{\sigma(j); \; \sigma \in A\} \text{ for each } j \in I$. 
	We write $\widetilde{A}=A-w_A$ and we define $\widetilde{\DD}=(\FF;\{\widetilde{\Delta}_J\}_{J \in \HH}, d)$.

\end{itemize}	
Let us consider a polyhedra system $\DD=(\FF;\{\Delta_J\}_{J \in \HH}, d)$ and a stratum $T \in \HH$. We introduce now a new polyhedra system $\DD^T$, using \emph{Hironaka's projection of $\DD$ from $T$}, that plays an important role in Section \ref{sec:MCF}. 

Let $\FF^T=(I^T,\HH^T)$ be the support fabric obtained by projection of $\FF$ from $T$. Given $J^* \in \HH^T$, let us take the stratum $J =J^* \cup T \in \HH$ and let us consider the subset $M_J^T \subset \RR^J$ given by $M_J^T=\{ \sigma \in \RR^J; \; \sum_{j \in T}\sigma(j) < 1\}$. \emph{Hironaka's projection} $\nabla_J^T: M_J^T \rightarrow \RR_{\geq 0}^{J^*}$ is defined by
$$\nabla_J^T(\sigma)= \dfrac{1}{1-\sum\limits_{j \in T}\sigma(j)}\sigma_{|J\setminus T}.$$
for every $\sigma \in M_J^T$.
We define $\Delta_{J^*}^T$ by $\Delta_{J^*}^T=\nabla_J^T(\Delta_J \cap M_J^T) \subset \RR^{J^*}_{\geq 0}$.
The reader can verify that
\begin{equation} \label{eq:proyeccion de Hironaka}
\DD^T=(\FF^T; \{\Delta_{J^*}^T\}_{J^* \in \HH^T},d!d),
\end{equation}
is a polyhedra system in lower dimension. It is important to remark that the denominator of $\DD^T$ is $d!d$ instead of $d$.

Note that $\Delta_{J^*}^T=\emptyset$ if $\Delta_J = \emptyset$ or $\sum_{j \in T}\sigma(j) \geq 1$ for all $\sigma \in \Delta_J$.  Note also that in a polyhedra system either all the polyhedra are the empty set or none of them is empty.
From now, we suppose (unless otherwise stated) that all the polyhedra $\Delta$ we are working with, are such that $\Delta \neq \emptyset$. 

\subsection{Singular Locus of Polyhedra Systems}
Let $\Delta \subset \RR^J_{\geq 0}$ be a characteristic polyhedron with $J \neq \emptyset$. The \emph{contact exponent $\delta(\Delta)$} is defined by 
$$\delta(\Delta)=\min \left\{|\sigma|;\; \sigma \in \Delta\right\}, \text{ where } |\sigma|= \sum\limits_{j \in J}\sigma(j).$$
When $J = \emptyset$, there is only one possible polyhedron $\Delta=\{\bullet \} \subset \RR^{\emptyset}_{\geq 0}$. By convention, we assume $\delta(\{\bullet \})=-1$.
\begin{rem}
If $\Delta=[[A]]$ with $A \subset \ZZ_{\geq 0}$ (Newton polyhedron), the contact exponent corresponds to the classical idea of multiplicity.
\end{rem}
Now, let $\DD=(\FF; \{\Delta_J\}_{J\in \HH},d)$ be a polyhedra system. The \emph{contact exponent $\delta(\DD)$} is defined by $\delta(\DD)=\max\{\delta(\Delta_J); \; J \in \HH\}$. The \emph{singular locus} $\text{Sing}(\DD)$ is the subset 
$$\text{Sing}(\DD)=\{J \in \HH; \; \delta(\Delta_J) \geq 1\} \subset \HH.$$ 
We have $\text{Sing}(\DD)$ is a closed set, since the contact exponent gives an upper semicontinuous function in $\HH$. We say that $\DD$ is \emph{singular} if $\text{Sing}(\DD)\neq \emptyset$, otherwise it is \emph{non-singular}.

\subsection{Transforms of Polyhedra Systems under Blow-ups}

Let us consider the blow-up $\pi_J(\FF)=(I'=I \cup \{\infty\},\HH'=\HH'_s \cup \HH'_{\infty})$ of a support fabric $\FF=(I,\HH)$ centered in a non-empty stratum $J\in \HH$. Given $J' \in \HH'_{\infty}$ and $K=\pi_J^{\#}(J')$, we define $\lambda_{J'}:\RR^K \rightarrow \RR^{J'}$ by  
$$\lambda_{J'}(\sigma)(j)=\sigma(j),  \; j \neq \infty; \quad \lambda_{J'}(\sigma)(\infty)=|\sigma_{|J}|=\sum\limits_{j\in J }\sigma(j).$$
Let us denote by $\{e_{I,i}\}_{i \in I}$ the standard basis of $\RR^I$, that is $e_{I,i}(i')=\delta_{i,i'}$ (Kronecker). If $J \subset I$ and $i \in I$ we define $e_{J,i} \in \RR^J$ by $e_{J,i}=\text{pr}_{I,J}(e_{I,i})$.

\begin{defin}
The \emph{total transform $\Lambda^0_J(\DD)$ of a polyhedra system $\DD=(\FF;\{\Delta_J\}_{J\in \HH},d)$ centered in a stratum $J \in \HH$} is the polyhedra system $\Lambda^0_J(\DD)=(\pi_J(\FF);\{\Delta^0_{J'}\}_{J'\in \HH'},d)$, where
	 $$\Delta^0_{J'}=\Delta_{J'}, \;J' \in \HH'_s; \quad \Delta^0_{J'}=\left[\left[ \lambda_{J'} (\Delta_K)\right]\right], \; J'\in \HH_{\infty}', \, K=\pi_J^{\#}(J').$$
If $J\in Sing(\DD)$, the \emph{characteristic transform $\Lambda_J(\DD)$} of $\DD$ \emph{centered in $J$} is defined by the polyhedra system $\Lambda_J(\DD)=(\pi_J(\FF);\{\Delta'_{J'}\}_{J'\in \HH'},d)$, where $\Delta'_{J'}=\Delta^0_{J'}-e_{J',\infty}$ for each $J' \in \HH'$.
\end{defin}
The characteristic transform $\Lambda_J(\DD)$ is a polyhedra system again, because the center of the blow-up is singular and then $\delta(\Delta^0_{\{\infty\}}) \geq 1$.

Let $\NN=(\FF; \{N_J\}_{J \in \HH},1)$ be a Newton polyhedra system and take an integer number $d \geq 1$. Consider the polyhedra system $\DD=\NN /d$ and a singular stratum  $J\in \text{Sing}(\DD)$. The \emph{d-moderated transform $\Theta^d_J(\NN)$ of $\NN$ centered in $J$} is defined by
$\Theta^d_J(\NN)=\NN(\Lambda_J(\DD))$.
Note that $J\in \text{Sing}(\DD)$ if and only if $\delta(N_J)\geq d$. 

\begin{rem}
	If there is an equivalence $\phi$ between  two polyhedra systems $\DD_1$ and $\DD_2$, then for all $J \in \text{Sing}(\DD_1)$, the characteristic transforms $\Lambda_J(\DD_1)$ and $\Lambda_{\phi(J)}(\DD_2)$ are also equivalent.	
\end{rem}

\section{Reduction of Singularities.}
\subsection {Statements}

We say that a Newton polyhedra system $\NN=(\FF;\{N_J\}_{J\in\HH},1)$ has \emph{normal crossings} if the polyhedron $N_J$ has a single vertex for each $J\in \HH$.

The property of having normal crossings is stable under blow-ups. A first objective of reduction of singularities is to get this property.

\begin{teo}[Combinatorial Reduction to Normal Crossings]\label{teo:NC} Given a Newton polyhedra system $\NN$, there is a finite sequence of total transforms 
$$\NN \rightarrow \NN_1 \rightarrow \NN_2 \rightarrow \cdots \rightarrow \NN_k$$
such that $\NN_k$ has normal crossings.
\end{teo}  
We can find proofs of this result in another contexts in \cite{Fer} and \cite{Gow}. Anyway, we provide a complete proof next.
\begin{teo} [Combinatorial Reduction of Singularities]\label{teo:ch}
Given a polyhedra system $\DD$, there is a finite sequence of characteristic transforms 
$$\DD \rightarrow \DD_1 \rightarrow \DD_2 \rightarrow \cdots \rightarrow \DD_k$$
such that $Sing(\DD_k)=\emptyset$. We call such a sequence a \emph{reduction of singularities of $\DD$}.
\end{teo}

\begin{rem}
Theorem \ref{teo:ch} would be false if we had taken the condition $\delta(\Delta_J) > 1$, instead of $\delta(\Delta_J) \geq 1$, for the centers of blow-up.
For instance, consider the local polyhedra system $\LL(\Delta,1)$ where $\Delta=[[(1,1)]] \subset \RR^{\{1,2\}}_{\geq 0}$. The only possible center would be $\{1,2\}$ and the situation is repeated in each closed strata of the characteristic transform.
\end{rem}
\begin{cor} \label{cor: teo2 implica teo1}
Let $\NN$ be a Newton polyhedra system and $d$ a positive integer number. There is a finite sequence of $d$-moderated transforms
$$\NN \rightarrow \NN_1 \rightarrow \NN_2 \rightarrow \cdots \rightarrow \NN_k$$
such that $\delta(\NN_k) < d$.
\end{cor}

Taking $d=1$, the reader can see that Corollary \ref{cor: teo2 implica teo1} implies Theorem \ref{teo:NC}. 

The objective now is to provide a proof of Theorem \ref{teo:ch}. Although this result follows from general Hironaka's reduction of singularities, we give a complete combinatorial proof here, emphasizing the ideas of maximal contact developed by Hironaka, Aroca and Vicente in \cite{Aro-H-V} as well as the polyhedra control suggested by Spivakosvky in \cite{Spi}. 

\subsection{Induction Procedure}
The proof of Theorem \ref{teo:ch} runs essentially by induction on the dimension $\text{dim} (\FF_\DD)$ of $\DD$. More precisely, let us consider the following statement.
\begin{quote}  CRS($n$):
		\em If $dim(\FF_\DD)\leq n$, then $\DD$ has a reduction of singularities.
\end{quote}
The starting step of the induction is given by 
\begin{prop}
	CRS($1$) holds.
\begin{proof}
Choose the invariant $I(\DD)$ given by $I(\DD)=\sum_{J \in \HH}\delta(\Delta_J)$. We have $I(\DD) \geq -1$ and after a single characteristic transform  drops exactly a unit. Thus, in a certain point of the process, we can no more perform a transform. We are done, since then the singular locus is empty.
\end{proof}
\end{prop}

We distinguish three types of singular polyhedra systems $\DD$:
\begin{enumerate}
	\item We say that $\DD$ is \emph{Hironaka quasi-ordinary} if each polyhedron over a singular stratum has a single vertex, that is, $\Delta_J=[[\{\sigma_J\}]]$ for all $J \in \text{Sing}(\DD)$.
	\item We say that $\DD$ is \emph{special} if $\delta(\DD) = 1$.
	\item We say that $\DD$ is \emph{general} if $\delta(\DD) > 1$.
\end{enumerate}

The study of Hironaka quasi-ordinary systems does not require the induction hypothesis as we see in Section \ref{sec:HQ-O}.

We consider special systems under the induction hypothesis CRS($n-1$). In this case we reduce the problem to lower dimension projecting over a \emph{maximal contact support fabric}. The existence of maximal contact goes as follows:
\begin{enumerate}
	\item We solve the local problem by considering Hironaka's strict tangent space. 
	\item We eliminate the possible ``global incoherence'' by invoking the induction hypothesis.
\end{enumerate} 
Details are given in Section \ref{sec: special systems}.

Finally, for general systems, we define Spivakovsky's invariant, that ``measures'' how far is the system from the quasi-ordinary case. We control the behaviour of this invariant by means of a special system. Details are given in Section \ref{sec:GPS}.

\section{Examples}
We consider a pair $(M,E)$, where $M$ is a complex analytic variety and $E$ is a strong normal crossings divisor as in Example \ref{ex:fabric}. Recall that we have a support fabric $\FF_{M,E}=(I,\HH)$ associated to $(M,E)$. Remark that there is a stratification of $M$ with strata $S_J$ given by 
$$S_J=E_J \setminus \bigcup_{i \notin J} E_i,$$
where $E_J= \bigcap_{j \in J} E_j\neq \emptyset$, for each $J \in \HH$.	

The followings are examples of polyhedra systems that motivate our definitions:

\subsection{Hypersurfaces}

Let $H \subset M$ be a closed hypersurface of $(M, \mathcal{O}_M)$ given by an invertible coherent ideal sheaf $\mathcal{I} \subset \mathcal{O}_M$ and let us suppose that $E_i \not\subset H$ for all $i \in I$.
% ¿Qué problema hay al suponer que algún E_i puede estar contenido en H?
We say that $H$ is \emph{combinatorially regular for $(M,E)$} if we have the property $E_J \not\subset H$, for all $J \in \HH$.

We can attach a polyhedra system $\NN_{M,E; H}$ to $H$ as follows.
Given $J \in \HH$, we consider a point $p \in S_J$ and a generator $F$ of $\mathcal{I}_p$. Let us write $F$ in local adapted coordinates at $p$ as 
$$F=\sum a_{\sigma}(\underline{y})\underline{x}^{\sigma}; \text{ with } \sigma \in \ZZ_{\geq 0}^J.$$
where the coefficients $a_{\sigma}(\underline{y})$ are germs of functions in $p$ defined in an open set of $S_J$. We define $N_J$ by
$$N_J=\left[\left[\{\sigma \in \ZZ_{\geq 0}^J; \; a_{\sigma}(\underline{y}) \not\equiv 0\}\right]\right].$$
The polyhedron $N_J$ is independent of $p$, the generator $F$ and the chosen local adapted coordinates. The independence of the local adapted coordinates and of the generator is straightforward. Once fixed a local coordinate system, we can move to points $p'$ close to $p$ in the same stratum and we obtain the same polyhedron in the translated coordinates. Finally, thanks to the connectedness of the stratum, we can join two points by a compact path and repeat finitely many times the above procedure.

This construction is compatible with the projections $p_{J_2J_1}:\RR^{J_2} \rightarrow \RR^{J_1}$, then we obtain the Newton polyhedra system $\NN_{M,E;H}$. 

The following statements are equivalent:
\begin{enumerate}
	\item The hypersurface $H$ is combinatorially regular for $(M,E)$.
	\item The Newton polyhedra system is non-singular, that is $\delta(N_J)=0$ for all $J \in \HH$.
\end{enumerate}
Let $\nu_J(H)$ be the generic multiplicity of $H$ along $E_J$. Note that $\nu_J(H)=\delta(N_J)$. We say that $H$ is \emph{combinatorially equimultiple along $E_J$} if $\nu_J(H) =\nu_K(H)$ for every $K\in \HH$ such that $E_K \subset E_J$.
Now, let us consider $J \in \HH$ such that $m=\delta(N_J)$ is maximal. In particular, we have $H$ is combinatorially equimultiple along $E_J$. We can perform the usual blow-up $\pi:(M',E') \rightarrow (M,E)$ centered in $E_J$. The Newton polyhedra system attached to the strict transform $H'$ of $H$ is given by the $m$-moderated transform of $\NN_{M,E; H}$.
Then, applying repeatedly Corollary \ref{cor: teo2 implica teo1} we obtain a finite sequence of blow-ups centered in combinatorially equimultiple strata.
$$
(M,E;H)=(M^0,E^0;H^0) \leftarrow (M^1,E^1;H^1) \leftarrow \cdots \leftarrow (M^n,E^n;H^n)=(M',E';H')
$$
so that $H'$ is combinatorially regular for $(M',E')$.

The polyhedra system of this example is particularly useful when we consider Newton non-degenerate functions as in \cite{Kou,Oka}.
%\subsection{Effective divisors}
%The Newton polyhedra system $\NN_{M,E; D}$ associated to an effective divisor $D=\sum n_iH_i$
% is defined in the same way as before by considering the local equations of $D$ at each point.

\subsection{Codimension one singular foliations}
Let us consider now a singular foliation $\LL$ of codimension one over $M$. We know that $\LL$ is given by an integrable and invertible coherent  $\mathcal{O}_M$-submodule of $\Omega^1_M[E]$, where $\Omega^1_M[E]$ denotes the sheaf of meromorphic one-forms having at most simple poles along $E$ (See \cite{Can,Can-C}).
Given a point $p \in S_J \subset M$, the stalk $\LL_p$ is generated by a meromorphic differential one-form $\omega \in \Omega^1_M[E]$, satisfying the Frobenius integrability condition $\omega \wedge d\omega=0$. Let us write $\omega$ in local adapted coordinates at $p$ as
$$\omega = \sum_{\sigma \in \ZZ^J_{\geq 0}}\omega_{\sigma}\underline{x}^{\sigma}, \text{ with } \omega_{\sigma}= \sum\limits_{i \in J}a_{\sigma,i}(\underline{y})\frac{dx_i}{x_i}+\sum b_{\sigma,s}(\underline{y})dy_s$$
where the coefficients $a_{\sigma}(\underline{y})$, $b_{\sigma}(\underline{y})$ are germs of functions at $p$, without common factor, defined in an open set of $S_J$. We have that $E_J$ is contained in the adapted singular locus $\text{Sing}(\LL,E)$ of $\LL$ when $\omega_{0}\equiv 0$. On the other hand, we say that $\LL$ is \emph{combinatorially regular for $(M,E)$} if $E_J \not\subset \text{Sing}(\LL,E)$ for every $J \in \HH$.

We attach a polyhedra system $\NN_{M,E; \LL}$ to $\LL$ as follows. Given $J \in \HH$, take a point $p \in S_J$ and a generator $\omega$ of $\LL_p$ as before. We define $N_J$ by
$$N_J=\left[\left[\{\sigma \in \ZZ_{\geq 0}^J; \; \omega_{\sigma} \not\equiv 0\}\right]\right].$$
The polyhedra $N_J$ only depend on $\LL$ and $S_J$, moreover, they are compatible with the projections $p_{J_2J_1}$. We obtain in this way the Newton polyhedra system $\NN_{M,E;\LL}$. 

The following statements are equivalent:
\begin{enumerate}
	\item The foliation $\LL$ is combinatorially regular for $(M,E)$.
	\item The Newton polyhedra system is non-singular, that is $\delta(N_J)=0$ for all $J \in \HH$.
\end{enumerate}
Let $\nu_J(\LL)$ be the adapted order of $\LL$ along $E_J$ (See \cite{Can}). Note that 
$$\nu_J(\LL)=\min\left\{\nu_{\underline{x}}\left(\sum x^{\sigma}a_{\sigma,i}(\underline{y})\right), \nu_{\underline{x}}\left(\sum x^{\sigma}b_{\sigma,s}(\underline{y})\right)\right\}_{i,s}=$$
$$=\min\left\{|\sigma| ; \; \sigma \in N_J\right\}=\delta(N_J).$$
As before, we say that $\LL$ is \emph{combinatorially equimultiple along $E_J$} if $\nu_J(\LL) =\nu_K(\LL)$ for every $K\in \HH$ such that $E_K \subset E_J$.
Considering $J \in \HH$ such that $m=\delta(N_J)$ is maximal, we perform the usual blow-up $\pi:(M',E') \rightarrow (M,E)$ centered in $E_J$ and the Newton polyhedra system attached to the strict transform $\LL'$ of $\LL$ is given by the $m$-moderated transform $\Theta^m(\NN_{M,E; \LL})$. As in the preceding example, we obtain a finite sequence of blow-ups centered in combinatorially equimultiple strata $(M',E';\LL')\rightarrow (M,E;\LL)$ such that $\LL'$ is combinatorially regular for $(M',E')$.

\begin{rem}
	
We can obtain an example similar to the preceding one in the case of one-dimensional singular foliations (See \cite{Can2,McQ-P,Pan}) that are locally generated by vector fields of the form
 $$\chi= \sum\chi_{\sigma}\underline{x}^{\sigma}, \text{ with } \chi_{\sigma}= \sum\limits_{i \in J}a_{\sigma,i}(\underline{y})x_i\frac{\partial}{\partial x_i}+\sum b_{\sigma,s}(\underline{y})\frac{\partial}{\partial y_s}.$$
\end{rem}

\subsection{Monomial ideals} 

Let $\mathcal{I}$ be a monomial ideal of $\mathcal{O}_M$ adapted to $E$. Following \cite{Gow}, we define $\mathcal{I}$ to be the data of a finite collection of effective divisors $\{D_s\}_{s \in S}$ of the form $$D_s=\sum_{i \in I} n_i^sE_i.$$ 
We attach a Newton polyhedra system $\NN_{M,E;\mathcal{I}}$ to $\mathcal{I}$ as follows. Let $\sigma^s:I \rightarrow \ZZ_{\geq 0}$ be the map given by $\sigma^s(i)=n^s_i$. For each $J \in \HH$, we define $N_J$ by 
$$N_J=[[\{\sigma^s|_J; \; s \in S\}]].$$
Note that $\NN_{M,E;\mathcal{I}}$ is non-singular if and only if $\mathcal{I}$ is a principal ideal.
Given $J \in \HH$, we can perform the usual blow-up $\pi:(M',E') \rightarrow (M,E)$ centered in $E_J$. The Newton polyhedra system attached to the total transform $\mathcal{I}'$ of $\mathcal{I}$ is given by the total transform of $\NN_{M,E; \mathcal{I}}$.
Now, we apply Theorem \ref{teo:NC} and, as in \cite{Gow}, we obtain a finite sequence of blow-ups $(M',E';\mathcal{I}') \rightarrow (M,E;\mathcal{I})$ such that $\mathcal{I}'$ is a principal ideal.

\section{Hironaka Quasi-Ordinary Polyhedra Systems} \label{sec:HQ-O}
In this section we prove Theorem \ref{teo:ch} for Hironaka quasi-ordinary polyhedra systems.

Given a Hironaka quasi-ordinary system $\DD=(\FF; \{\Delta_{J}\}_{J \in \HH},d)$ and a singular stratum $J$, the transform $\Lambda_J(\DD)=(\pi_J(\FF); \{\Delta'_{J'}\}_{J' \in \HH'},d)$ is also Hironaka quasi-ordinary. Indeed, for every $J'\in \HH'_{\infty}$ and $K=\pi_J^{\#}(J')$, we have $\Delta_{J'}=[[\{\sigma_{J'}\}]]$, with $\sigma_{J'}=\lambda_{J'}(\sigma_K)-e_{J',\infty}$.

\begin{lema} \label{sucesiones}
	The set $\mathcal{A}$ of decreasing sequences of natural numbers is well-ordered for the lexicographical order.
	\begin{proof}	    
		Given a decreasing sequence $\varphi_1 \geq \varphi_2 \geq \cdots \geq \varphi_k \geq \cdots$ in $\mathcal{A}$, we take $k_0=0$ and $k_i=\min \{j \geq k_{i-1}; \; \varphi_{j'}(i)=\varphi_j(i), \, j' \geq j\} < \infty$, for each $i \geq 1$. Let us consider the decreasing sequence $\varphi:\mathbb{N}\rightarrow \mathbb{N}$ given by $\varphi(n)=\varphi_{k_n}(n)$. There is a positive number $l \in \mathbb{N}$ such that $\varphi(n)=\varphi(l)$, for all $n \geq l$. As a consequence we have $\varphi_j=\varphi_{k_l}$, for all $j \geq k_l$.
	\end{proof}       
\end{lema} 

\begin{prop} \label{prop:qo}
	Every Hironaka quasi-ordinary polyhedra system has reduction of singularities.
\end{prop}
\begin{proof}
	Choose a bijective map $s_\DD:\{1,2,\ldots,\#\HH\}\rightarrow \HH$ such that 
	$$\delta(\Delta_{s_\DD(j)}) \geq \delta(\Delta_{s_\DD(j+1)}); \; j \geq 1.$$
	Take the decreasing sequence $\varphi: \mathbb{N} \rightarrow \mathbb{N}$ given by $\varphi(j)=d\delta(\Delta_{s_\DD(j)})$ if $j \leq \#\HH$ and $\varphi(j)= 0$ if $j > \#\HH$. We choose $\varphi$ as lexicographical invariant. Let us see that after an appropriate single blow-up the invariant $\varphi$ drops and then we are done by Lemma \ref{sucesiones}. We select as center a singular stratum $J$ with the minimum number of elements. Given $J' \in \HH'_{\infty}$ and $K=\pi_J^{\#}(J')$, we have
	$$\delta(\Delta'_{J'})=|\sigma_{J'}|=|\lambda_{J'}(\sigma_K)|-1=|\sigma_K|+|\sigma_{A_{J'}}|-1 < \delta(\Delta_K)$$
	and then we obtain  $\varphi' < \varphi$.
\end{proof}
\section{Special Polyhedra Systems} \label{sec: special systems}
In this section we prove Theorem \ref{teo:ch} for the case of special polyhedra systems. More precisely we prove {\bf CRS($n$)} for special polyhedra systems under the induction hypothesis {\bf CRS($n-1$)}.
\subsection{Stability of Special Systems}
A special polyhedra system is either non-singular or special under blow-up. Precisely, we have

\begin{prop} \label{prop:estabilidad especiales}
	Let $\DD$ be a special polyhedra system. Given a singular stratum $J$, we have $\delta(\Lambda_J(\DD)) \leq 1$.
	\begin{proof}
	Let us denote $\DD=(\FF;\{\Delta_J\}_{J \in \HH},d)$ and $\Lambda_J(\DD)=(\pi_J(\FF);\{\Delta'_{J'}\}_{J' \in \HH'},d)$.
	Take $J' \in \HH_{\infty}'$ and $K=\pi_J^{\#}(J')$. We have to see that $\delta(\Delta'_{J'}) \leq  \delta(\Delta_K)=1$. Given $\sigma\in \Delta_K$ with $|\sigma|=1$, we have $\delta(\Delta'_{J'}) \leq |\lambda_{J'}(\sigma)-e_{J',\infty}|=|\sigma|-1+|\sigma_{|A_{J'}}| \leq 1$.
	\end{proof}
\end{prop}

\begin{rem}
Let us consider a Newton polyhedra system $\NN$ and an integer $d \geq 1$. A stratum $J \in \HH$ is \emph{$d$-equimultiple} if the polyhedra system $\text{Red}_{\overline{\{J\}}\cap \HH}(\NN/d)$ is special. Proposition \ref{prop:estabilidad especiales} means in this context that the multiplicity is stable under $d$-moderated transforms when we perform blow-ups centered in $d$-equimultiple strata.
\end{rem}
\subsection{Hironaka's Strict Tangent Space} \label{sec:STS}

Given a characteristic polyhedron $\Delta \subset \RR^I_{\geq 0}$, with $\delta(\Delta)=1$, we recall that \emph{Hironaka's strict tangent space $T(\Delta)$} of $\Delta$, is defined by
$$T(\Delta)=\{j \in I; \;\exists \, \sigma\in \Delta, \, |\sigma|=1, \, \sigma(j) \neq 0\} \subset I.$$
See \cite{Aro-H-V,Hir,Hir2,Spi}. In this section we present the clasical stability results concerning the behaviour of the strict tangent space in a ``horizontal'' and a ``vertical'' way.

Let us consider a special polyhedra system $\DD=(\FF; \{\Delta_J\}_{J\in \HH},d)$.
\begin{prop}\label{pr:Estab.hor.}
	We have $T(\Delta_{J_2}) \subset T(\Delta_{J_1}) \subset J_1$, for every $J_1,J_2 \in \text{Sing}(\DD)$ with $J_1 \subset J_2$.
	\begin{proof}      
		Given $\sigma \in \Delta_{J_2}$ such that $|\sigma|=1$, the restriction $\sigma_{|J_1} \in \Delta_{J_1}$ also satisfies $|\sigma_{|J_1}|=1$ because $\delta(\Delta_{J_1})=1$. Then $\sigma(j)=0$ for all $j \in J_2\setminus J_1$. Thus, we have both $T(\Delta_{J_2}) \subset J_1$ and $T(\Delta_{J_2}) \subset T(\Delta_{J_1})$.
	\end{proof}        
\end{prop}         
\begin{rem}        
	In the situation of Proposition \ref{pr:Estab.hor.}, we can have $T(\Delta_{J_2}) \ne T(\Delta_{J_1})$. For instance, let us consider  $\Delta=[[(0,0,1),(1,1,0)]] \subset \RR^{\{1,2,3\}}_{\geq 0}$. In this case, we obtain 
	$$\{3\}=T(\Delta) \neq T(\Delta_{\{2,3\}})=\{2,3\}.$$
\end{rem}          
Consider a singular stratum $J \in \text{Sing}(\DD)$ and let us perform the characteristic transform $\DD'=\Lambda_J(\DD)=(\pi_J(\FF); \{\Delta'_{J'}\}_{J'\in \HH'},d)$ of $\DD$ centered in $J$. By Proposition \ref{prop:estabilidad especiales}, we know that $\DD'$ is either non-singular or a special system.

\begin{prop}  \label{prop:vstab sts}
	Consider a singular stratum $J'=(K \setminus J) \cup A_{J'} \cup \{\infty\}\in \HH'_{\infty}$ where $K=\pi_J^{\#}(J')$. We have $T(\Delta_K)\subset A_{J'} \subset J$. Moreover, $T(\Delta_K)\subset T(\Delta'_{J'}) \subset J'$.
	\begin{proof} 
	Given $\sigma \in \Delta_K$ with $|\sigma|=1$, we have $\delta(\Delta'_{J'}) \leq |\lambda_{J'}(\sigma)-e_{J',\infty}|=|\sigma_{|A_{J'}}| \leq 1$.
	As a consequence $|\sigma_{|A_{J'}}|=1$ or equivalently $T(\Delta_K)\subset A_{J'}$. Now, given $j \in T(\Delta_K)$ we have $j\in A_{J'}$. As $|\sigma'|=1$ and $\sigma'(j)=\sigma(j)\neq 0$, then $j\in T(\Delta'_{J'})$.
	\end{proof}        
\end{prop}

\begin{rem}
	The property $T(\Delta_K) \subset A_{J'}$ can be read as saying that ``the new singular strata are in the transform of the strict tangent space''. This is the main classical feature of Hironaka's strict tangent space.
\end{rem}

\begin{cor} \label{cor:espacio tangente}
	Given $J'\in \HH'_{\infty}$ and $K=\pi_J^\#(J')$, if $T(\Delta_K)=J$ then $J' \notin \text{Sing}(\DD')$.
	\begin{proof}
	If $J' \in \text{Sing}(\DD')$ by Proposition \ref{prop:vstab sts}, we have $T(\Delta_K) \subset A_{J'}$, but $A_{J'} \subsetneq J$.
	\end{proof}
\end{cor}

\subsection{Reduction of Singularities via Maximal Contact} \label{sec:MCF}
In this section we obtain reduction of singularities for special polyhedra systems under the assumption that there is maximal contact with a given stratum.

Let us consider a special polyhedra system $\DD=(\FF;\{\Delta_J\}_{J \in \HH},d)$ and a stratum $T \in \HH$. We say that $T$ has \emph{maximal contact with $\DD$} if $T \subset T(\Delta_J)$, for all $J \in \text{Sing}(\DD)$. We say also that $\FF^T$ is a \emph{maximal contact support fabric of $\DD$}. 
\begin{rem} \label{rk: T singular }
	Let $T \in \HH$ be a stratum that has maximal contact with $\DD$. Recall that Hironaka's projection of $\DD$ over $T$ is given by $\DD^T=(\FF^T; \{\Delta_{J^*}^T\}_{J^* \in \HH^T},d!d)$. Since we are assuming that the polyhedra of $\DD$ are non-empty, we have $\Delta_{J^*}^T=\emptyset$ for all $J \in \HH^T$ if and only if $T \in \text{Sing}(\DD)$. Note that in this case $\Lambda_T(\DD)$ is non-singular as a consequence of Corollary \ref{cor:espacio tangente}.
\end{rem}

\begin{rem} \label{rk: contacto maximal local}
For every local polyhedra system $\DD=\LL(\Delta,d)$, each non-empty stratum $T \subset T(\Delta)=T(\Delta_I)$ has maximal contact with $\DD$. That happens because of the horizontal stability of the strict tangent space stated in Proposition \ref{pr:Estab.hor.}. Indeed, for every singular stratum $J \in \text{Sing}(\DD)$, we have $J \subset I$ and hence $T(\Delta_I)\subset T(\Delta_J)$. 
\end{rem}

\begin{lema} \label{lema: estabilidad contacto maximal} Let us consider the transform $\Lambda_J(\DD)$ where $J \in \text{Sing}(\DD)$. If $T$ has maximal contact with $\DD$, then $T$ also has maximal contact with $\Lambda_J(\DD)$.
\begin{proof}
	 Recall that $\pi_J(\FF)=(I',\HH'=\HH'_s \cup \HH'_{\infty})$. Note that $T \subset J$ with $T \neq J$ and then $T \in \HH'_s$. The maximal contact of $T$ with $\Lambda_J(\DD)$ is assured by the property of vertical stability for the strict tangent space given in Proposition \ref{prop:vstab sts}.
\end{proof}	
\end{lema}

\begin{prop} \label{prop: maximal contact}
	Let $T \notin \text{Sing}(\DD)$ be a stratum that has maximal contact with $\DD$ and let us consider Hironaka's projection $\DD^T$ of $\DD$ from $T$. If there is a reduction of singularities for $\DD^T$, then $\DD$ also has a reduction of singularities.
\end{prop}

Recall that $\DD^T=(\FF^T; \{\Delta^T_{J^*}\}_{J^* \in \HH^T},d!d)$ has been introduced in equation (\ref{eq:proyeccion de Hironaka}).
 We prove of the proposition in three steps.

	\emph{Step 1:  There is a one-to-one correspondence between}  $\text{Sing}(\DD)$ \emph{and} $\text{Sing}(\DD^T)$. Let us denote by $\HH_T$ the set $\HH_T=\{J \in \HH; \; T \subset J\}$. Note that $\HH_T$ is an open set of $\HH$. In fact, is the smallest one containing $\overline{\{T\}} \cap \HH$. There is a bijection $\Psi: \HH_T\rightarrow \HH^T$ given by $\Psi(J)=J \setminus T$.	Note that $\text{Sing}(\DD) \subset\HH_T$. 
	We claim that $\Psi(\text{Sing}(\DD))=\text{Sing}(\DD^T)$. Given $J \in \text{Sing}(\DD)$, we have to prove that $|\sigma^*| \geq 1$ for all $\sigma^* \in \Delta^T_{J\setminus T}$. Given such a $\sigma^*$, there is $\sigma \in \Delta_J \cap M_J^T$, satisfying 
	$$\sigma^*=\frac{\sigma|_{J\setminus T}}{1-|\sigma|_T|}.$$ 
	Note that $|\sigma|_T| < 1$, since $\sigma \in M^T_J$. Recalling that $T \subset J$, we have 
	\begin{equation} \label{eq:sigma estrella}
	|\sigma^*|=\frac{|\sigma_{|{J \setminus T}}|+|\sigma|_T|-|\sigma|_T|}{1-|\sigma|_T|}=\frac{|\sigma|-|\sigma|_T|}{1-|\sigma|_T|}.
	\end{equation}
	Then $|\sigma^*| \geq 1$ since $|\sigma| \geq 1$. As a consequence $J \setminus T \in \text{Sing}(\DD^T)$.
	Conversely, if we take $J^* \in \text{Sing}(\DD^T)$, we have to prove that $|\sigma| \geq 1$, for all $\sigma \in \Delta_J$, where $J=J^* \cup T$. Given such a $\sigma$, we have two possibilities. If $|\sigma_{|T}| \geq 1$ we are done. Otherwise $|\sigma_{|T}| < 1$ and then $\sigma \in M_J^T$. In this case, there is $\sigma^* \in \Delta^T_{J^*}$ given by $\sigma^*=\sigma_{|J\setminus T}/(1-|\sigma_{|T}|)$. By equation (\ref{eq:sigma estrella}), we have $(1-|\sigma_{|T}|)|\sigma^*|=|\sigma|-|\sigma_{|T}|$. Moreover $1-|\sigma_{|T}| \leq |\sigma|-|\sigma_{|T}|$ since $|\sigma^*| \geq 1$. As a consequence $|\sigma| \geq 1$ and then $J \in \text{Sing}(\DD)$.

	\emph{Step 2: Commutativity of Hironaka's projections with the blow-up transforms}. We are going to show that 
	$$(\Lambda_J(\DD))^T = \Lambda_{J \setminus T} (\DD^T).$$ 
	We verify this commutativity first for the support fabrics and second for each polyhedron.
	 \begin{enumerate}
	 	\item [a)] \emph{Commutativity for the support fabrics.} 
	 	Let us see that $\pi_J(\FF)^T=\pi_{J \setminus T}(\FF^T)$. Denote 
	 	$$\pi_J(\FF)^T=\Big((I')^T,(\HH')^T\Big), \quad \pi_{J \setminus T}(\FF^T)=\Big((I^T)',(\HH^T)'\Big).$$
	 	The index sets are equal, since $(I')^T=(I \cup \{\infty\})\setminus T = (I \setminus T) \cup \{\infty\}=(I^T)'$.
	  	
	  	Let us see that $(\HH'_s)^T=(\HH^T)'_s$. Note that $K \supset J \Leftrightarrow K \setminus T \supset J \setminus T$ for avery $K \in \HH$, because $T \subset J$. Then, for each $J'^*$, we have
	  	$$J'^* \in (\HH'_s)^T \Leftrightarrow J'^*\cup T \in \HH'_s, \, J'^*\cap T= \emptyset \Leftrightarrow  J'^* \in (\HH^T)'_s.$$	  
		Take $K \in \HH$ with $K \supset J$. Let us see that $(\HH'^K_{\infty})^T=(\HH^T)'^{K \setminus T}_{\infty}$. For each $J'^*$, we have
		$$J'^* \in (\HH'^K_{\infty})^T \Rightarrow J'=J'^*\cup T \in \HH'^K_{\infty} \Rightarrow T \subset A_{J'} \subsetneq J \Rightarrow $$
		$$ \Rightarrow  A_{J'} \setminus T \subsetneq J \setminus T  \Rightarrow J'^* \in (\HH^T)'^{K \setminus T}_{\infty}.$$
		
		$$J'^* \in (\HH^T)'^{K \setminus T}_{\infty} \Rightarrow A_{J'^*} \subsetneq J \setminus T \Rightarrow A = A_{J'^*} \cup T \subsetneq J\Rightarrow$$ $$\Rightarrow J'^*\cup T \in \HH'^K_{\infty}  \Rightarrow J'^* \in (\HH^T)'^{K \setminus T}_{\infty} .$$
		 	
		\item [b)] \emph{Commutativity for the polyhedra.} It is enough to prove that the following diagram commutes
		$$\xymatrix{ M_K^T \cap \RR^K_{\geq 0} \ar[r]^{f_{J'}} \ar[d]_{\nabla^T_K}  & \RR^{J'}_{\geq 0} \cap M_{J'}^T \ar[d]^{\nabla^T_{J'}} \\ \RR^{K\setminus T} \ar[r]^{f_{J' \setminus T}} & \RR^{J' \setminus T}_{\geq 0}}$$	
		where  $J' \in \HH'_{\infty}$, $K= \pi_J^{\#}(J')$, $T \subset J'$, $T \subset K$  and $f_{J'}(\sigma)=\lambda_{J'}(\sigma)-e_{J',\infty}$.  
		Note that $f_{J'}$ is well-defined because $f_{J'}(\sigma)_{|T}=\sigma_{|T}$ for all $\sigma \in \RR^K_{\geq 0}$.

		Given $\sigma \in  M_K^T \cap \RR^K_{\geq 0}$, we denote 
		$$\sigma'=f_{J'}(\sigma), \; \sigma^T=\nabla^T_K(\sigma), \; (\sigma^T)'=f_{J'\setminus T}(\sigma^T), \; (\sigma')^T=\nabla^T_{J'}(\sigma').$$
		We conclude that $(\sigma^T)'=(\sigma')^T$ from the following equalities: 	
		$$(\sigma')^T_{|J' \setminus (T \cup \{\infty\})}=\frac{\sigma_{|J' \setminus (T \cup \{\infty\})}}{1-|\sigma_{|T}|}=(\sigma^T)'_{|J' \setminus (T \cup \{\infty\}};$$
		$$(\sigma')^T(\infty)=\frac{|\sigma_{|J}|-1}{1-|\sigma_{|T}|}=\frac{|\sigma_{|J\setminus T}|+|\sigma_{|T}|-1}{1-|\sigma_{|T}|}=\frac{|\sigma_{|J\setminus T}|}{1-|\sigma_{|T}|}-1=(\sigma^T)'(\infty).$$
	 \end{enumerate}
Hence, $\DD$ has reduction of singularities. This ends the proof of Proposition \ref{prop: maximal contact}.
 
\subsection{Connected Components Decomposition of Desingularization} \label{sec:CP}
The problem of reduction of singularities can be solved by considering one by one each connected component of the singular locus.

\begin{prop} \label{prop: componentes conexas}
Let $\DD$ be a polyhedra system and $\text{Sing}(\DD)= \bigcup\CC_{\alpha}$ the decomposition in connected components of the singular locus. Assume that there is reduction of singularities for each $\text{Red}_{\CC_{\alpha}}(\DD)$. Then there is reduction of singularities for $\DD$.
\end{prop}
In order to simplify notation, we denote $\DD_{\alpha}=\text{Red}_{\CC_{\alpha}}(\DD)$ and $\FF_{\alpha}$ the support fabric of $\DD_{\alpha}$. Let us consider a connected component $\CC_{\alpha}$. By hypothesis, $\DD_{\alpha}$ has a reduction of singularities 
$$
\DD_{\alpha} \rightarrow \DD_{\alpha}^1 \rightarrow \cdots \rightarrow \DD_{\alpha}^k.
$$
This sequence induces a sequence of characteristic transforms 
$$
\DD \rightarrow \DD^1 \rightarrow \cdots \rightarrow \DD^k,
$$
obtained performing blow-ups in the same centers. Let $\FF^k$ be the support fabric of $\DD^k$. The polyhedra system $\DD^k$ satisfies the following properties: 
	\begin{enumerate}
		\item For each $\beta \neq \alpha$, there is an equivalence between $\DD_{\beta}$ and $\text{Red}_{\CC_{\beta}}(\DD^k)$ given by $\text{Id}_{\mathcal{I}}$, where $\mathcal{I}=\mathcal{I}_{\FF_{\beta}}=\mathcal{I}_{\FF^k_{\beta}}$, with $\FF^k_{\beta}=\text{Red}_{\CC_{\beta}}(\FF^k)$.
		\item The  decomposition of the singular locus $\text{Sing}(\DD^k)$ in connected components is given by $\text{Sing}(\DD^k)= \bigcup_{\beta \neq \alpha}\CC_{\beta}$.
	\end{enumerate}
Proof of Proposition \ref{prop: componentes conexas} is concluded applying the same argument finitely many times. 
				
\subsection{Existence of Maximal Contact}
\begin{defin} \label{def: consistente}
Let $\DD$ be a special polyhedra system and let $\text{Sing}(\DD)= \bigcup\CC_{\alpha}$ be the decomposition of the singular locus 
in connected components. We say that $\DD$ is \emph{consistent} if for each connected component $\CC_{\alpha}$, there is a stratum $T_{\alpha}$ that has maximal contact with $\text{Red}_{\CC_{\alpha}}(\DD)$. By convention, non-singular polyhedra systems are consistent.
\end{defin}

\begin{prop} \label{prop:crs special} Under the induction hypothesis CRS($n-1$), every $n$-dimensional consistent polyhedra system has reduction of singularities.
	\begin{proof}
		Let $\DD$ be a $n$-dimensional consistent polyhedra system. Let us consider the decomposition of the singular locus in connected components $\text{Sing}(\DD)=\bigcup\CC_{\alpha}$. In order to simplify notation, we write $\DD_{\alpha}=\text{Red}_{\CC_{\alpha}}(\DD)$. For each $\CC_{\alpha}$, we have a stratum $T_{\alpha}$ that has maximal contact with $\DD_{\alpha}$. If $T_{\alpha} \in \text{Sing}(\DD)$, by Remark \ref{rk: T singular }, $\DD_{\alpha}$ has reduction of singularities. Otherwise, since Hironaka's projection $\DD_{\alpha}^{T_{\alpha}}$ has dimension lower than $n$, there is reduction of singularities for $\DD_{\alpha}^{T_{\alpha}}$ and by Proposition \ref{prop: maximal contact} we obtain reduction of singularities for $\DD_{\alpha}$. By Proposition \ref{prop: componentes conexas} we conclude the proof.
	\end{proof}
\end{prop}

\begin{rem}
There exist non-consistent polyhedra systems. For instance, consider the support fabric $\FF=(I,\HH)$ with $I=\{1,2,3,4\},$
$$\HH=P(I)\setminus \big\{\{1,4\},\{1,2,4\},\{1,3,4\},I\big\}.$$ 
Take over it the polyhedra system $\DD=(\FF; \{\Delta_J\}_{J\in \HH},d)$, given by 
$$\Delta_{\{1,2,3\}}=\big[\big[(0,0,1),(1/2,1,0)\big]\big], \, \Delta_{\{2,3,4\}}=\big[\big[(1,0,0),(0,1,1/2)\big]\big]$$ 
and all their projections. We have $\text{Sing}(\DD)=\big\{ \{2,3\}, \{1,2,3\}, \{2,3,4\}\big\}$ is connected. 
We have that $T(\Delta_{\{1,2,3\}})=\{3\}$ and $T(\Delta_{\{2,3,4\}})=\{2\}$, then there is no stratum $T$ with maximal contact.
\end{rem}

The following statement says that every non-consistent special system of dimension $n$ can be transformed in a finite number of steps in a consistent one. More precisely

\begin{prop} \label{prop: hacer consistente}
	Under the induction hypothesis  CRS($n-1$), for every $n$-dimensional special polyhedra system $\DD$, there is a finite sequence of characteristic transforms $\DD \rightarrow \DD'$ such that $\DD'$ is consistent.
\begin{proof}
	
The main idea in this proof is to reduce the problem to the case of polyhedra systems obtained by blow-up of local systems.The existence of maximal contact in those systems is assured by Remark \ref{rk: contacto maximal local} and by the stability results in Lemma \ref{lema: estabilidad contacto maximal}.	
	
Let us consider the set $\KK$ of closed strata in $\HH$ and the open set $\UU=\HH \setminus \KK$.
Recall that $\KK$ is a finite, totally disconnected closed set.
We have $\text{dim}(\DD_{|\UU}) < n$, since the dimension is reached only at closed points. Then $\DD_{|\UU}$ has a reduction of singularities
$$\DD_{|\UU} \rightarrow \DD_{|\UU}^1 \rightarrow \cdots \rightarrow \DD_{|\UU}^k.$$
This sequence gives rise to a sequence of characteristic transforms 
\begin{equation} \label{eq: sucesion global}
\DD \rightarrow \DD^1 \rightarrow \cdots \rightarrow \DD^k
\end{equation}
obtained performing blow-ups in the same centers. Let us prove that $\DD^k$ is consistent. 

Write $\DD_{|\UU}^k=\left(\FF_{|\UU}^k, \{\Delta^k_{J'}\}_{J' \in \HH_{|\UU}^k},d\right)$ and $\DD^k=\left(\FF^k,\{\Delta^k_{J'}\}_{J' \in \HH^k}, d\right)$.
By Corollary \ref{cor:continuidad}, we know that $\FF_{|\UU}^k=\FF^k|_{\UU_k}$ where $\UU_k=(\pi^{\#}_k)^{-1}(\UU)$ and $\pi^{\#}_k:\HH^k \rightarrow \HH$ is given by the sequence of equation (\ref{eq: sucesion global}).
We have $\text{Sing}(\DD^k) \cap \UU_k = \emptyset$ since $\DD_{|\UU}^k$ is non-singular. Moreover, we have $\HH^k=\UU_k \cup (\pi^{\#}_k)^{-1}(\KK)$. As a consequence, the singular locus $\text{Sing}(\DD^k)$ is contained in $(\pi^{\#}_k)^{-1}(\KK)$, that we write as
$$\left(\pi^{\#}_k\right)^{-1}(\KK)= \bigcup_{J \in \KK}\KK_J, \text{ where } \KK_J=\left(\pi^{\#}_k\right)^{-1}(J).$$
Since $\KK$ is totally disconnected, we have $\KK_{J_1} \cap \KK_{J_2}= \emptyset$ for every $J_1, J_2 \in \KK$ with $J_1 \neq J_2$. We conclude that, given a connected component $\CC$ of $\text{Sing}(\DD^k)$, there is a unique stratum $J \in \KK$ such that $\CC \subset \KK_J$.
On the other hand, we have $T(\Delta_J) \subset T(\Delta^k_{J'})$ for every $J' \in \CC$, because of the vertical stability property of the strict tangent space stated in Proposition \ref{prop:vstab sts}. Then the non-empty stratum $T(\Delta_J)$ has maximal contact with $\text{Red}_{\CC}(\DD_k)$.
\end{proof}
\end{prop}
\begin{cor}
	Under the induction hypothesis  CRS($n-1$), there is a reduction of singularities for every $n$-dimensional special polyhedra system.
\end{cor}

\section{General Polyhedra Systems} \label{sec:GPS}

In this section we consider the case of $n$-dimensional general polyhedra systems. In order to prove it, we assume the hypothesis 
\begin{description}
	\item CRS$^{\text{sp}}$($n$): $n$-dimensional special polyhedra systems have reduction of singularities.
\end{description}	

Note that  CRS($n-1$) implies CRS$^{\text{sp}}$($n$). The main idea is to transform a given general polyhedra system into a Hironaka quasi-ordinary one, using the hypothesis  CRS$^{\text{sp}}$($n$) and many of the ideas in Spivakovsky's work \cite{Spi}. Then we are done in view of Proposition \ref{prop:qo}.

\subsection{Spivakovsky's Invariant}
Let $\DD=(\FF;\{\Delta_J\}_{J\in \HH},d)$ be a polyhedra system. We recall that the fitting of $\DD$ is given by $\widetilde{\DD}=(\FF;\{\widetilde{\Delta}_J\}_{J \in \HH}, d)$, where $\widetilde{\Delta_J}=\Delta_J-w_{\Delta_J}$. We define \emph{Spivakovsky's invariant} $\text{Spi}_\DD$ by
$$\text{Spi}_\DD = \max\left\{\delta(\widetilde{\Delta_J}); \; J \in \text{Sing}(\DD)\right\}.$$
Note that $\text{Spi}_\DD = 0$ if and only if $\DD$ is a Hironaka quasi-ordinary polyhedra system.

We define the set $S(\DD)$ by $S(\DD)=\left\{J\in \text{Sing}(\DD); \; \delta(\widetilde{\Delta_J}) = \text{Spi}_\DD \right\}$.
\begin{lema}

	For every $J \in S(\DD)$, we have $\text{Spi}_{\Lambda_J(\DD)} \leq \text{Spi}_\DD$.
	\begin{proof}
	 Let us denote $\Lambda_J(\DD)=(\FF'; \{\Delta'_{J'}\}_{J' \in \HH'},d)$ and $\lambda=\text{Spi}_\DD$. We write for simplicity $w_K$ instead of $w_{\Delta_K}$. We need to prove that 
	 $$\delta(\widetilde{\Delta'_{J'}}) \leq \delta(\widetilde{\Delta_K}); \quad J' \in \HH', \, K=\pi_J^\#(J').$$  
	 Note that $\delta(\widetilde{\Delta_K})=\delta(\Delta_K)-|w_K|$. If $J' \in \HH'_s$ we are done. If $J' \in \HH'_{\infty}$, then we have $w_{J'}(j)=w_K(j)$ if $j \neq \infty$ and $w_{J'}(\infty)=\delta(\Delta_J)-1$. Moreover, $K \in S(\DD)$, since $K \supset J$.
	 Take $\sigma \in \Delta_K$ such that $|\sigma-w_K|=\delta(\widetilde{\Delta_K})$. Then
		$$\delta(\widetilde{\Delta'_{J'}}) \leq |\lambda_{J'}(\sigma)-e_{J'\infty}-w_{J'}|=|\sigma|+|\sigma_{|A_{J'}}|-1-(|w_{|J' \setminus \{\infty\}}|)+(\delta(\Delta_J)-1)=$$
		$$=\delta(\widetilde{\Delta_K})+|w_K| +|\sigma_{|A_{J'}}|-|w_{(K\setminus J)\cup A_{J'}}|-\delta(\widetilde{\Delta_J})-|w_J|= 
		\lambda + |\sigma_{|A_{J'}}| - |w_{A_{J'}}|- \lambda=$$
		$$=|\sigma_{|A_{J'}}-w_{A_{J'}}| \leq |\sigma| = \delta(\widetilde{\Delta_K}).$$
	\end{proof}
\end{lema}

\subsection{Resolution of General Systems}
Here we prove that Spivakovsky's invariant decreases strictly after a ``well-chosen'' finite sequence of blow-ups. We do it 
by means of Spivakovsky's projection introduced below and the hypothesis CRS$^{\text{sp}}$($n$).

Let  $\DD=(\FF;\{\Delta_J\}_{J \in \HH},d)$ be a singular polyhedra system with $\text{Spi}_{\DD} \neq 0$. We define \emph{Spivakovsky's projection $\DD^{sp}$} by  
$$\DD^{sp}=(\FF; \{\Delta_J^{sp}\}_{J \in \HH},d^2), \text{ where } \Delta_J^{sp}=\left[\left[\widetilde{\Delta_J}/\text{Spi}_\DD \cup \Delta_J\right]\right].$$
We have that $\DD^{sp}$ is a special polyhedra system and $\text{Sing}(\DD^{sp})=S(\DD)$ since  
$$\delta(\Delta_J^{sp})=\min \left\{\delta\left(\widetilde{\Delta_J}/\text{Spi}_{\DD}\right), \delta (\Delta_J)\right\} \leq 1.$$

\begin{prop} \label{prop: crs gen}
Under the hypothesis CRS$^{\text{sp}}$($n$), there is a finite sequence of characteristic transforms 
$$
\DD=\DD^0 \rightarrow \DD^1 \rightarrow \cdots \rightarrow \DD^k
$$ 
such that $\text{Spi}_{\DD^k} < \text{Spi}_\DD$ and the center $J^i$ of the characteristic transform $\DD^{i} \rightarrow \DD^{i+1}$ belongs to $S(\DD^i)$ for each $i=0, 1, \ldots, k-1$. 
	\begin{proof}
		We consider $J \in S(\DD)$ and denote $\lambda =\text{Spi}_\DD$ . We have two situations: 
		\begin{enumerate}
			\item  $\text{Spi}_{\Lambda_J(\DD)} < \lambda$. (In this case, we are done).
			\item  $\text{Spi}_{\Lambda_J(\DD)} = \lambda$.
		\end{enumerate}
		
		Suppose that $\text{Spi}_{\Lambda_J(\DD)} = \lambda$. Let us prove the commutativity of Spivakovsky's projections with the blow-up transforms. More precisely, given a stratum $J' \in \HH'_{\infty}$ and $K= \pi_J^{\#}(J')$, we want to see that
		\begin{equation} \label{eq:caso general}
		\left[\left[ \widetilde{[[f_{J'}(\Delta_K)]]}/\lambda \cup [[f_{J'}(\Delta_K)]]  \right]\right]=
		\left[\left[ f_{J'}\left(\left[\left[{\widetilde{\Delta_K}}/{\lambda} \cup \Delta_K\right]\right]\right)\right]\right],
		\end{equation}		
		where $f_{J'}: \RR^K_{\geq 0} \rightarrow \RR^{J'}_{\geq 0}$ is defined by $f_{J'}(\sigma)=\lambda_{J'}(\sigma)-e_{J',\infty}$.
		That is, recalling the definition of $\lambda_{J'}$,  $f_{J'}(\sigma)(j)=\sigma(j)$ if $j \neq \infty$ and $f_{J'}(\sigma)(\infty)= |\sigma_{|J}|-1$. 

It is not difficult to see that, the proof of equality (\ref{eq:caso general}) is reduced to prove the equality ${\widetilde{f_{J'}(\Delta_K)}}/{\lambda}=f_{J'}\left({\widetilde{\Delta_K}}/{\lambda}\right)$. That is, we want to show that
$$\frac{(\sigma'-w_{J'})}{\lambda}=f_{J'}\left(\frac{\sigma-w_K}{\lambda}\right), \text{ where } \sigma \in \Delta_K, \; \sigma'=f_{J'}(\sigma).$$
Note that we have $w_{J'}(j) = w_{K}(j)$ if $j \neq \infty$. Then we just need to see what happens with $j=\infty$. Recall that $w_{J'}(\infty)=\delta(\Delta_J)-1$, $\delta(\Delta_J)=\delta(\widetilde{\Delta_J})+ |w_J|$ and $\delta(\widetilde{\Delta_J})=\lambda$. Then
$$\frac{(\sigma'-w_{J'})(\infty)}{\lambda}= \frac{|\sigma_{|J}|-(s+|w_J|)}{\lambda}=\frac{|\sigma_{|J}-w_J|}{\lambda} -1 =
f_{J'}\left(\frac{\sigma-w_K}{\lambda}\right)(\infty).$$

By CRS$^{\text{sp}}$($n$), there is a reduction of singularities of $\DD^{sp}$. Then the second situation cannot be repeated forever and we are done.
\end{proof}
\end{prop}

Now the proof of Theorem \ref{teo:ch} is completed.

\vspace{2mm}
\textbf{Acknowledgements} I would like to express my gratitude to Professor Felipe Cano, for all the time devoted to the supervision of this work. I am also grateful with all the referees suggestions that have improved the text.

The author is partially supported by the Ministerio de Educaci\'on, Cultura y Deporte of Spain (FPU14/02653 grant) and by the Ministerio de Econom\'ia y Competitividad from Spain, under the Project ``Algebra y geometr\'ia en sistemas din\'amicos y foliaciones singulares.'' (Ref.:  MTM2016-77642-C2-1-P).

\end{document}